\documentclass[
12pt,reqno]{amsart}
\usepackage{verbatim}
\usepackage{amssymb}

\usepackage{enumerate}
\usepackage[active]{srcltx}
\numberwithin{equation}{section}

\usepackage{t1enc}

\newtheorem{theorem}{Theorem}[section]

\newtheorem{corollary}[theorem]{Corollary}

\newtheorem{problem}[theorem]{Problem}

\theoremstyle{definition}
\newtheorem{definition}[theorem]{Definition}

\theoremstyle{remark}

\newcommand{\itemprefix}{}
\makeatletter
\newcommand{\myitem}{%
\item\protected@edef\@currentlabel{\itemprefix\theenumi}%
}
\makeatother

\newcommand{\setm}{\setminus}

\newcommand{\subs}{\subset}

\def\con{\subseteq}

\def\reals{\mathbb{R}}
\def\from{\colon}

\author[I. Juh\'asz]{Istv\'an Juh\'asz}
\address      { Alfr\'ed Rényi Institute of Mathematics%
}
\email{juhasz@renyi.hu}

\author[J. van Mill]{Jan van Mill}
\address{University of Amsterdam}
\email{j.vanMill@uva.nl}

\thanks{The first author was supported
by NKFIH grant no. K129211.}

\makeatletter
  \@namedef{subjclassname@2020}{%
  \textup{2020} Mathematics Subject Classification}
   \makeatother


\title{Nowhere constant families of maps and resolvability}

\begin{document}

\begin{abstract}
If $X$ is a topological space and $Y$ is any set then we call a family $\mathcal{F}$ of maps from $X$
to $Y$ {\em nowhere constant} if for every non-empty open set $U$ in $X$ there is
$f \in \mathcal{F}$ with $|f[U]| > 1$, i.e. $f$ is not constant on $U$. We prove the following result
that improves several earlier results in the literature.

If $X$ is a topological space for which $C(X)$, the family of all continuous maps of $X$ to $\mathbb{R}$, is
nowhere constant
and $X$ has a $\pi$-base consisting of connected sets then
$X$ is $\mathfrak{c}$-resolvable.
\end{abstract}

\maketitle

\section{Introduction}

The question about how resolvable are crowded locally connected spaces has been around
for some time. C. Costantini proved in \cite{C} that {\em regular} such spaces are $\omega$-resolvable.
Actually, it is proved there that local connected can be weakened to having a $\pi$-base consisting of
connected sets. We simply call such spaces \emph{$\pi$-connected}.

In \cite{P} it is stated that that I.V. Yaschenko proved that every crowded locally connected
Tykhonov space is $\mathfrak{c}$-resolvable, however, as far as we know, no proof of this has
been published.

A. Dehghani and M. Karavan claim in \cite{DK}
that every crowded locally connected functionally Hausdorff space is
$\mathfrak{c}$-resolvable. They get this as a corollary of their more general Theorem 2.6,
in the proof of which, however, we found a gap,
and we do not know if this gap can be fixed. On Page 88 of their paper, lines -7 and -6, they claim that 
$\bigcup_{\gamma\not=\alpha} A_\gamma = f^{-1}(\bigcup_{\gamma\not=\alpha} D_\gamma)$ is not closed 
in any nonempty connected open subset $V$ of $U$. But this can only be concluded in case the restriction of $f$ to $V$ is not constant,
and they only assume that some continuous function $g : V \to \mathbb{R}$ is not constant.

It follows, however, from our results below that their Corollary~2.7 of Theorem 2.6 is correct.
In fact, the following property weaker than being crowded and functionally Hausdorff suffices: for every
non-empty open set $U$ there is a continuous map of the whole space to $\mathbb{R}$
that is not constant on $U$.

We are going to get this from a more general result that will make use of the following concept.

\begin{definition}
Let $X$ be a topological space and $Y$  any set. We call a family $\mathcal{F}$ of maps from $X$
to $Y$ {\em nowhere constant} (in short: $NWC$) if for every non-empty open set $U$ in $X$ there is
$f \in \mathcal{F}$ with $|f[U]| > 1$, i.e. $f$ is not constant on $U$.
\end{definition}

We note the trivial fact that any space that admits a NWC family of maps is crowded,
i.e. has no isolated points.

\section{The results}\label{results}

We first present a very general result that connects NWC families of maps and resolvability.

\begin{theorem}\label{tm:gen}
Let $\mathcal{F}$ be a NWC family of maps of the topological space $X$ to the set $Y$.
Moreover, $\mathcal{B}$ is a $\pi$-base of $X$
and $\mathcal{A}$ is a {\em disjoint} family of subsets of $Y$ such that, putting
$\mathcal{B}_f = \{B \in \mathcal{B} : |f[B]| > 1\}$ and $U_f = \bigcup (\mathcal{B} \setm \mathcal{B}_f)$
for any $f \in \mathcal{F}$, we have
$$\forall\,f \in \mathcal{F}\,\forall\,B \in \mathcal{B}_f\,\forall\,A \in \mathcal{A} \,(A \cap f[B \setm U_f] \ne \emptyset).$$
Then $X$ is $|\mathcal{A}|$-resolvable.
\end{theorem}

\begin{proof}
We are going to use transfinite recursion to produce a disjoint collection $\{D(A) : A \in \mathcal{A}\}$
of dense subsets of $X$.

To do that, we first introduce the following piece of notation. For $f \in \mathcal{F},\,\,B \in \mathcal{B}_f$, and
$A \in \mathcal{A}$ we let
$$
S(f, B, A) = f^{-1}[A] \cap (B \setm U_f).
$$
It follows from our assumptions that $S(f, B, A) \ne \emptyset$.

We start our recursive construction by chosing $f_0 \in \mathcal{F}$ such that $\mathcal{B}_{f_0} \ne \emptyset$
and then define $$D_0(A) = \bigcup \{S(f_0,B,A) : B \in \mathcal{B}_{f_0}\}$$ for all $A \in \mathcal{A}$.
Then $\{D_0(A) : A \in \mathcal{A}\}$ is disjoint because $D_0(A) \subs {f_0}^{-1}[A]$ for each $A \in \mathcal{A}$,
moreover $B \cap D_0(A) \ne \emptyset$ whenever $B \in \mathcal{B}_{f_0}$ and $A \in \mathcal{A}$.

Now, assume that $\alpha > 0$ and we have already defined $f_\beta \in \mathcal{F}$ and the family $\{D_\beta(A) : A \in \mathcal{A}\}$
for all $\beta < \alpha$, and consider $$\mathcal{B}_\alpha = \bigcup \{\mathcal{B}_{f_\beta} : \beta < \alpha\}.$$
If $\mathcal{B}_\alpha = \mathcal{B},$ then our construction stops.
Otherwise, we may pick $f_\alpha \in \mathcal{F}$ such that $\mathcal{B}_{f\alpha} \setm \mathcal{B}_\alpha \ne \emptyset$.
In this case we define $$D_\alpha(A) = \bigcup \{D_\beta(A) : \beta < \alpha\} \cup \bigcup\{S(f_\alpha, B, A) : B \in \mathcal{B}_{f\alpha} \setm \mathcal{B}_\alpha\}$$
for each $A \in \mathcal{A}$.

Note that we clearly have $\bigcup (\mathcal{B}_{f\alpha} \setm \mathcal{B}_\alpha) \subs \bigcap \{U_{f_\beta} : \beta < \alpha\}$,
hence it may be verified by straightforward transfinite induction that $\{D_\alpha(A) : A \in \mathcal{A}\}$ is
disjoint, moreover we have $B \cap D_\alpha(A) \ne \emptyset$ whenever $A \in \mathcal{A}$ and $B \in \mathcal{B}_{f_\beta}$ with $\beta \le \alpha$.

Of course, this construction must stop at some ordinal $\alpha$, in which case putting $D(A) = \bigcup \{D_\beta(A) : \beta < \alpha\}$
the disjoint family $\{D(A) : A \in \mathcal{A}\}$ consists of dense subsets of $X$ because
$B \cap D(A) \ne \emptyset$ for any $B \in \mathcal{B}$ and $A \in \mathcal{A}$ and $\mathcal{B}$ is a $\pi$-base of $X$.
\end{proof}

Now we deduce from Theorem \ref{tm:gen} our promised result concerning $\pi$-connected spaces.

\begin{corollary}\label{co:picon}
If $X$ is any $\pi$-connected space for which $C(X)$, the family of all continuous maps of $X$ to $\mathbb{R}$,
is NWC then $X$ is $\mathfrak{c}$-resolvable.
\end{corollary}

\begin{proof}
To apply Theorem \ref{tm:gen}, we of course put $Y = \mathbb{R}$ and $\mathcal{F} = C(X)$. For $\mathcal{B}$
we choose a $\pi$-base of $X$ consisting of connected sets and let $\mathcal{A}$ be any disjoint collection of
{\em dense} subsets of $\mathbb{R}$ with $|\mathcal{A}| = \mathfrak{c}$.

Note that for any $f \in C(X)$ and $B \in \mathcal{B}_f$ we have $|f[B]| > 1$, hence $f[B]$ is a non-degenerate
interval in $\mathbb{R}$, being a non-singleton connected subset of $\mathbb{R}$, and thus $A \cap f[B] \ne \emptyset$
holds for all $A \in \mathcal{A}$. Consequently, we shall be done if we prove the following claim.

\smallskip \noindent {\bf Claim}. For any $f \in C(X)$ and $B \in \mathcal{B}_f$ we have $$f[B] = f[B \setm U_f].$$
To see this, pick any $t \in f[B]$ and note that $B \cap f^{-1}(t)$ is a non-empty proper closed subset
of the connected subspace $B$. Properness follows from $|f[B]| > 1$. Consequently, the boundary $H_t$ of
$B \cap f^{-1}(t)$ in $B$ is non-empty. Also, we have $H_t \subs B \cap f^{-1}(t)$, the latter set being closed in $B$.

Now, take any $B' \in \mathcal{B} \setm \mathcal{B}_f$, then $f[B'] = \{t'\}$ for some $t' \in \mathbb{R}$.
If $t' \ne t$ then we have $B' \cap f^{-1}(t) = \emptyset$, hence $B' \cap H_t = \emptyset$ as well.
If, on the other hand, $t' = t$ then $B'\cap B$ is in the interior of $B \cap f^{-1}(t)$ in $B$,
hence we have $B' \cap H_t = \emptyset$  again. This, however, means that $H_t \cap U_f = \emptyset$,
consequently, $t \in f[H_t] \subs f[B \setm U_f]$, completing the proof of the Claim.
\end{proof}

A closer inspection of this proof reveals that we did not use the full force of continuity for the
members $f$ of the NWC family $\mathcal{F} = C(X)$. What we used was that $f$ preserves connectedness
for $B \in \mathcal{B}$,
i.e. $f[B]$ is connected in $\mathbb{R}$ for $B \in \mathcal{B}$, moreover that $f^{-1}(t)$ is closed in
$X$ for all $t \in \mathbb{R}$.

\section{Discussion and questions}

Corollary \ref{co:picon} is clearly sharp in the sense that the cardinal $\mathfrak{c}$
cannot be replaced with anything bigger
because there are very nice crowded and locally connected spaces of cardinality $\mathfrak{c}$.
The natural question arises, however, if such a space $X$ is maximally resolvable, i.e.
$\Delta(X)$-resolvable, where $\Delta(X)$ is the minimum cardinality of a non-empty
open set in $X$.
(
In a recent arXiv preprint \cite{L} A. Lipin proved that if $2^\mathfrak{c} = 2^{(\mathfrak{c}^+)}$
then there is a locally connected and pseudocompact Tykhonov space $X$ such that $\Delta(X) > \mathfrak{c}$ but
$X$ is not $\mathfrak{c}^+$-resolvable. So, at least consistently, Corollary \ref{co:picon} is
sharp in that sense as well.

As we wrote in the introduction, Costatini~\cite{C} proved that crowded $\pi$-connected regular spaces. 
are $\omega$-revolvable. As usual, in his treatment regular implies $T_1$.
But his proof actually works for all crowded $\pi$-connected {\em quasi-regular} spaces as well, 
without the use of any separation axiom.

Recall that a space is quasi-regular if for every non-empty open $U$ there is a non-empty open $V$
such that $\overline{V} \subset U$. 
The only slight modification we need is that in this case the definition of {\em crowded} has to be replaced by
the following assumption: there is no finite indiscrete open subspace. For $T_0$-spaces this assumption is clearly
equivalent with the usual one, i.e. not having any isolated points. 

Hewitt~\cite{H} gave an example of a regular connected space $X$ of cardinality~$\omega_1$. Let $Y=\lambda_f(X)$ be the 
superextension of $X$ consisting of all finitely generated maximal linked systems consisting of closed subsets of $X$. 
Then $Y$ also has size $\omega_1$, since the collection of all finite subsets of $X$ has size $\omega_1$,
moreover $Y$ is Hausdorff, connected and locally connected, see \cite[IV.3.4(v) and (viii)]{V}. 

We claim that $Y$ is also regular and adopt the terminology of~\cite{V}. To prove this, let $m\in Y$ be arbitrary, 
and let $A$ be a closed subset of $Y$ not containing $m$. Let $F\con X$ be a finite defining set for $m$, and let $n = m{\restriction}F$. 
There is a finite collection of open subsets $\mathcal{U}$ of $X$ such that $m\in \bigcap_{U\in \mathcal{U}} U^+ \con Y\setminus A$. 
For each $U\in \mathcal{U}$, let $N_U\in n$ be such that $N_U\con U$. By regularity of $X$, we may pick an open neighborhood 
$V_U$ of $N_U$ whose closure is contained in $U$. Then $\bigcap_{U\in \mathcal{U}} V_U^+$ is a closed neighborhood of $m$ that misses $A$. 

Hence $Y$ is a connected, locally connected regular space which by Costantini's result is $\omega$-resolvable. (In this special case, 
there is also a direct proof of this). But $Y$ has cardinality $\omega_1$, hence it is not $\mathfrak{c}$-resolvable if the Continuum Hypothesis fails. 
So, for crowded locally connected regular spaces a positive answer to the following question is the best we can hope for.

\begin{problem}
Let $X$ be regular, crowded and locally connected (respectively, $\pi$-connected). Is $X$ $\omega_1$-resolvable?
\end{problem}

We cannot resist mentioning here that
it is still widely open whether non-singleton connected regular spaces are resolvable, see e.g. \cite{P1} for details and references. 
This seems to be one of the most central open problems in the area.

We call a space $X$ \emph{nowhere 0-dimensional} if no non-empty open subspace of $X$ is §-dimensional.
Clearly, any crowded locally connected, even $\pi$-connected, space is nowhere 0-dimensional.

Assume that $X$ is a nowhere 0-dimensional Tychonoff space. We may assume that $X$ is a subspace of $\reals^I$ for some set $I$. 
For every $i\in I$, let $\pi_i\from \reals^I \to \reals$ be the $i$th projection. Then for every nonempty open subset $U$ of $X$
there exists $i\in I$ such that $\pi_i(U)$ has nonempty interior. For otherwise $U$ would be a subspace
of a product of 0-dimensional spaces, and so would itself be 0-dimensional.
Hence the collection of projections $\mathcal{P} = \{\pi_i{\restriction}X : i \in I\}$ is NWC in a (very) strong sense. 
This observation lead us to the following problem.

\begin{problem}
Is every nowhere 0-dimensional Tychonoff space resolvable ($\omega$-resolvable, $\mathfrak{c}$-resolvable)?
\end{problem}

\end{document}